\documentclass[11pt]{amsart}

\usepackage[utf8]{inputenc} 
\usepackage{amsthm,amstext,amsmath,amssymb}
\usepackage[unicode=true]{hyperref}
\usepackage[foot]{amsaddr}

\usepackage[margin=1.25in]{geometry}

\usepackage{graphicx}

\hypersetup{
    pdftitle={Improved bound for the dimension of posets of treewidth two},
    pdfauthor={Michał T. Seweryn}
}

\makeatletter
\let\old@setaddresses\@setaddresses
\def\@setaddresses{\bigskip\bgroup\parindent 0pt\let\scshape\relax\old@setaddresses\egroup}
\makeatother

\author[M.T.~Seweryn]{Michał T. Seweryn}
\address[M.T.~Seweryn]{Theoretical Computer Science Department\\
  Faculty of Mathematics and Computer Science, Jagiellonian University, Krak\'ow, Poland}
\email{michal.seweryn@tcs.uj.edu.pl}
\thanks{Michał Seweryn was partially supported by the National Science Center of Poland under grant no.\ 2015/18/E/ST6/00299.}

\newtheorem{theorem}{Theorem}
\newtheorem{lemma}[theorem]{Lemma}
\newtheorem{claim}[theorem]{Claim}
\newtheorem{observation}[theorem]{Observation}

\newcounter{inlineclaim}[theorem]
\renewcommand{\theinlineclaim}{(\arabic{inlineclaim})}

\makeatletter
\newcommand{\nameditem}[1]{%
\item[#1]\protected@edef\@currentlabel{#1}%
}
\makeatother

\newcommand{\bfP}{\mathbf{P}}
\newcommand{\bfS}{\mathbf{S}}

\newcommand{\lein}{\preceq}
\newcommand{\ltin}{\prec}

\DeclareMathOperator{\Inc}{Inc}

\begin{document}
\title{Improved bound for the dimension of posets of treewidth two}

\begin{abstract}
Joret et al.\ proved that posets with cover graphs of treewidth at most \(2\) have dimension at most \(1276\).
Their proof is long and very complex.
We give a short and much simpler proof that the dimension of such posets is at most \(12\).
\end{abstract}
\maketitle
\section{Introduction}

Partially ordered sets (posets for short) are one of the most often studied structures in combinatorics.
Perhaps the most important notion of complexity of posets is dimension.
Recall that the \emph{dimension} of a finite poset \(P\) is the least non-negative integer \(d\) such that the partial order in \(P\) is the intersection of \(d\) linear orders.
For \(n \ge 2\), 
the \emph{standard example} of order \(n\) is a poset \(S_n\) of height \(2\) with \(n\) minimal elements \(a_1, \dots, a_n\) and \(n\) maximal elements \(b_1, \dots, b_n\), such that for \(1 \le i \le n\) and \(1 \le j \le n\), we have \(a_i < b_j\) if and only if \(i \neq j\). The dimension of \(S_n\) is equal to \(n\).

Posets are visualized by their diagrams.
A \emph{diagram} of a poset \(P\) is a drawing on the plane, where each element of \(P\) is represented by a point, and each comparability \(x < y\) which is not
implied by transitivity is represented by a curve that goes upward from \(x\) to \(y\).
The diagram of a poset seen as an abstract graph is its cover graph.
More formally, for two elements \(x\) and \(y\) of \(P\), we say \(y\) \emph{covers} \(x\) if \(x < y\) in \(P\) but there is no element \(z\) of \(P\) such that \(x < z < y\) in \(P\). The \emph{cover graph} of \(P\) is a (simple) graph, whose vertex set is the ground set of \(P\) and two vertices are joined with an edge if and only if one of them covers the other in \(P\).

Connections between dimension of posets and structure of their cover graphs have been studied over years.
There is a common belief, that posets with a sufficiently ``sparse'' cover graph should have a small dimension.
For example, when the cover graph of a poset is a forest, then its dimension is at most \(3\)~\cite{trotter-moore}, while posets with outerplanar cover graphs have dimension at most \(4\)~\cite{felsner-trotter-wiechert}.
However, a famous construction by Kelly~\cite{kelly} gives for every \(n \ge 2\) a poset with planar cover graph, which contains the standard example \(S_n\) as a subposet (see Figure~\ref{fig:kelly}). This construction shows that posets with planar cover graphs can have arbitrarily large dimension.
\begin{figure}
  \includegraphics{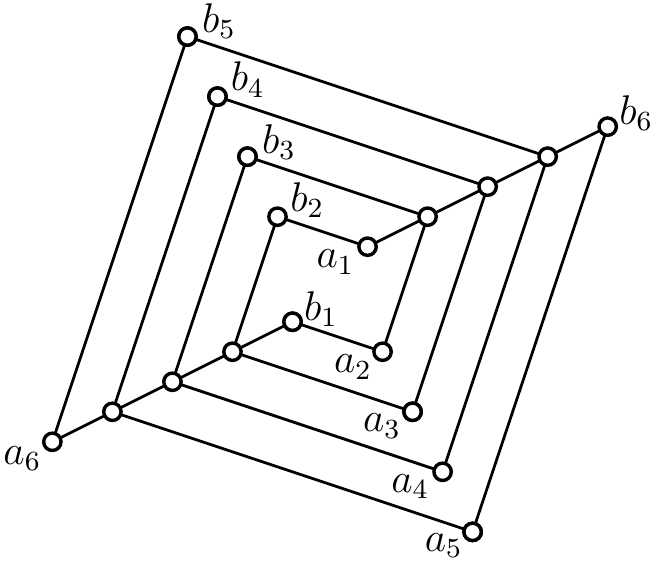}
  \caption{Kelly's construction of a poset of dimension \(6\)}\label{fig:kelly}.
\end{figure}



Bir\'o, Keller, and Young~\cite{biro-keller-young} showed that posets with cover graphs of pathwidth at most \(2\) have dimension at most \(17\).
This bound was later improved to \(6\) with a short proof by Wiechert~\cite{wiechert}.
On the other hand, posets with cover graphs of pathwidth (and treewidth) \(3\) can have arbitrarily large dimension, as witnessed by posets in Kelly's construction.
The question whether posets with cover graphs of treewidth at most \(2\) have bounded dimension was answered affirmatively by Joret et al.\ \cite{joret-et-al}.
They gave a complex and technical proof that every poset with a cover graph of treewidth at most \(2\) has dimension at most \(1276\).
The main contribution of this paper is a much simpler proof improving this bound.

\begin{theorem}\label{thm:main}
  Every finite poset whose cover graph has treewidth at most \(2\) has dimension at most \(12\).
\end{theorem}

Let us briefly describe the idea of the proof.
Our proof is inspired by the proof by Felsner, Trotter and Wiechert~\cite{felsner-trotter-wiechert} that posets with outerplanar cover graphs have dimension at most \(4\).
In their proof, the set of all pairs of incomparable elements of a poset is partitioned into \(4\) sets, and the partition heavily uses the linear order of elements along the outer face of the cover graph.
These \(4\) sets give rise to \(4\) linear orders, which witness that the dimension is at most \(4\).

To generalize the proof from the outerplanar case, we use a characterization of graphs of treewidth \(2\) as subgraphs of series-parallel graphs.
Series-parallel graphs have natural partial orderings of the vertices, and if a series-parallel graph is outerplanar, then this partial ordering can be chosen so that it coincides with the order of vertices on the outer face in some outerplanar drawing.

The recursive definition of series-parallel graph is not very convenient to work with, so instead using it directly in our proof, we use it to construct a special tree-decomposition, in which every bag has two distinguished vertices, which satisfy some additional properties.
Such a special tree-decomposition can be easily constructed without explicitly mentioning series-parallel graphs, but constructing it from the definition of series-parallel graphs allows to notice analogies to the proof for outerplanar cover graphs.

The bound \(12\) in our proof significantly improves the one from~\cite{joret-et-al}.
It is likely still not optimal because the largest known dimension of a poset with cover graph of treewidth \(2\) is \(4\).
However, we think that improving that bound is not of that much importance, and that the main advantage of our proof is its simplicity. 
At this point it seems more interesting to improve the lower bound for the optimal constant in Theorem~\ref{thm:main}. 
The best known lower bound \(4\) is reached by a poset with an outerplanar cover graph~\cite{felsner-trotter-wiechert}.
It seems that the class of outerplanar graphs is much more restricted than the class of graph of treewidth at most \(2\), and we believe that there should exist a poset with cover graph of treewidth \(2\) of dimension greater than \(4\).

\section{Poset preliminaries}

All posets considered in this paper are finite.
Let \(P\) be a poset.
For every element \(x\) of \(P\), the \emph{upset} of \(x\) in \(P\), denoted \(U_P(x)\), is the set of all elements \(y\) of \(P\) such that \(y \ge x\) in \(P\).
Similarly, the \emph{downset} of \(x\) in \(P\), denoted \(D_P(x)\), is the set of all elements \(y\) of \(P\) such that \(y \le x\) in \(P\).
In both cases, we drop subscript \(P\) when the poset is clear from the context.
The \emph{dual} of \(P\) is the poset \(P^d\) with the same ground set as \(P\), such that \(x \le y\) in \(P^d\) whenever \(y \le x\) in \(P\).
Clearly, \(P\) and \(P^d\) have the same cover graph and dimension.

A \emph{chain} in \(P\) is a set of linearly ordered elements of \(P\).
For two comparable elements \(x\) and \(y\) of \(P\) such that \(x \le y\) in \(P\), we define a \emph{covering chain} from \(x\) to \(y\) as a chain consisting of points \(x_1, \dots, x_k\) such that \(x_1 = x\), \(x_k = y\), and for every \(i \in \{2, \dots, k\}\), the point \(x_{i}\) covers \(x_{i-1}\) in \(P\).
Every covering chain from \(x\) to \(y\) induces a path between \(x\) and \(y\) in the cover graph of \(P\).
Every upset \(U(x)\) is a union of covering chains from \(x\) to the elements of \(U(x)\).
Hence every upset (and, by a dual argument, every downset) induces a connected subgraph of the cover graph.

An \emph{incomparable pair} in \(P\) is an ordered pair \((x, y)\) of incomparable elements in \(P\).
The set of all incomparable pairs in \(P\) is denoted by \(\Inc(P)\).
A \emph{linear extension} of \(P\) is a linear order \(L\) on the ground set of \(P\) such that \(x \le y\) in \(L\) for any two elements \(x\) and \(y\) of \(P\) such that \(x \le y\) in \(P\). 
A set \(I \subseteq \Inc(P)\) is \emph{reversible} in \(P\) if there exists a linear extension \(L\) of \(P\) such that \(y < x\) in \(L\) for every \((x, y) \in I\). The set \(I\) is reversible in \(P\) if and only if the set \(I^{-1} = \{(y, x) : (x, y) \in I\}\) is reversible in \(P^d\).
Here is a common and useful rephrasing of the definition of poset dimension.

\begin{observation}\label{obs:reversible-sets-and-dimension}
  If the dimension of a poset \(P\) is at least \(2\), then it is equal to the least positive integer \(d\) such that \(\Inc(P)\) can be partitioned into \(d\) reversible sets.
\end{observation}

For \(n \geq 2\), an indexed family \(\{(x_i, y_i) : i \in \{1, \dots, n\}\}\) of incomparable pairs in \(P\) is an \emph{alternating cycle} in \(P\) if \(x_i \leq y_{i+1}\) for \(i\in\{1,\dots,n\}\) (in alternating cycles we always interpret indices cyclically, so \(x_n \leq y_1\) is required).
An alternating cycle \(\{(x_i, y_i) : i \in \{1, \dots, n\}\}\) is \emph{strict} if for any pair of indices \(i\in \{1, \dots, n\}\) and \(j \in \{1, \dots, n\}\), we have \(x_i \leq y_j\) if and only if \(j = i+1\) (cyclically, as always).
The following is an easy and well-known characterization of reversible sets in terms of strict alternating cycles, originally observed by Trotter and Moore~\cite{trotter-moore}.

\begin{observation}\label{obs:reversible-sets-characterization}
  A set \(I \subseteq \Inc(P)\) is reversible in \(P\) if and only if there is no strict alternating cycle in \(P\) with all pairs from \(I\).
\end{observation}

\section{Treewidth \(2\) and series-parallel graphs}\label{sec:series-parallel}

All graphs considered in this paper are finite and simple. A \emph{tree-decomposition} of a graph \(G\) is a pair \((T, \{Y_u : u \in V(T)\})\), where \(T\) is a tree and \(\{Y_u : u \in V(T)\}\) is an indexed family of subsets of \(V(G)\) called \emph{bags} which satisfies the following conditions:
\begin{enumerate}
  \item[(T1)] \(V(G) = \bigcup_{u \in V(T)} Y_u\),
  \item[(T2)] for every edge \(xy\) of \(G\) there exists a node \(u\) of \(T\)
  such that \(\{x, y\} \subseteq Y_u\), and
  \item[(T3)] for every three nodes \(u_1\), \(u_2\) and \(v\) of \(T\), if \(v\) lies on the path between \(u_1\) and \(u_2\) in \(T\), then \(Y_{u_1} \cap Y_{u_2} \subseteq Y_{v}\).
\end{enumerate}

By property (T3), for every vertex \(x\) of \(G\), the nodes \(u\) such that \(x \in Y_u\) induce a subtree of \(T\).
In particular, if the tree \(T\) is rooted, then among all nodes \(u\) such that \(x \in Y_u\) there exists a unique one closest to the root.

The \emph{width} of a tree-decomposition \((T, \{Y_u : u \in V(T)\})\) is \[\max \{|Y_u| - 1 : u \in V(T)\}.\]
The \emph{treewidth} of a graph \(G\) is the least width of a tree-decomposition of \(G\).

The following is a well-known property of tree-decompositions:

\begin{observation}\label{obs:tree-decomposition-separations}
  Let \((T, \{Y_u : u \in V(T)\})\) be a tree-decomposition of a graph \(G\), let \(u_1\) and \(u_2\) be two nodes of \(T\), let \(v_1 v_2\) be an edge on the path between \(u_1\) and \(u_2\) in \(T\), and let \(H\) be a connected subgraph of \(G\).
  If \(V(H) \cap Y_{u_1} \neq \emptyset\) and \(V(H) \cap Y_{u_2} \neq \emptyset\), then \(V(H) \cap (Y_{v_1} \cap Y_{v_2}) \neq \emptyset\).
\end{observation}


A \emph{two-terminal graph} is a triple \((G, s, t)\), where \(G\) is a graph, \(s\) is a vertex of \(G\) called \emph{source}, and \(t\) is a vertex of \(G\) called \emph{sink}.
Let \((G_1, s_1, t_1)\) and \((G_2, s_2, t_2)\) be two-terminal graphs such that \(E(G_1) \cap E(G_2) = \emptyset\).
If \(V(G_1) \cap V(G_2) = \{t_1\} = \{s_2\}\), then the \emph{series composition}  of \((G_1, s_1, t_1)\) and \((G_2, s_2, t_2)\) is the two-terminal graph \[
\bfS((G_1, s_1, t_1), (G_2, s_2, t_2)) = (G_1 \cup G_2, s_1, t_2).\]
If \(s_1 = s_2 = s\), \(t_1 = t_2 = t\) and \(V(G_1) \cap V(G_2) = \{s, t\}\), then the \emph{parallel composition} of \((G_1, s_1, t_1)\) and \((G_2, s_2, t_2)\) is the two-terminal graph \[\bfP((G_1, s_1, t_1), (G_2, s_2, t_2)) = (G_1 \cup G_2, s, t).\]
\begin{figure}
  \includegraphics{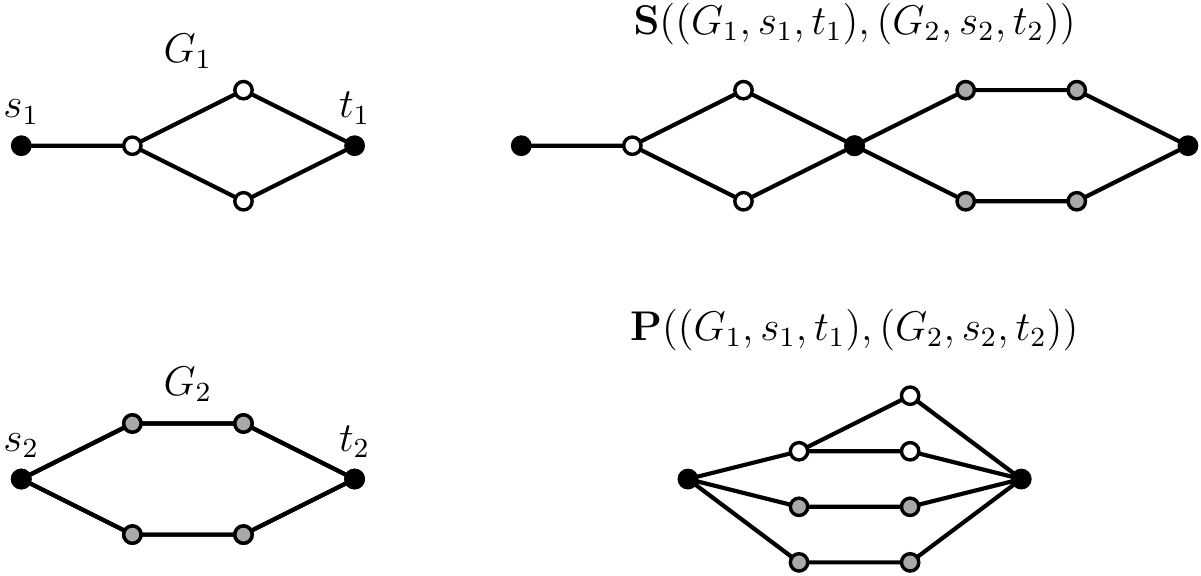}
  \caption{Series and parallel compositions of two-terminal graphs.}\label{fig:series-parallel}.
\end{figure}

A \emph{basic two-terminal series-parallel graph} is a two-terminal graph \((G, s, t)\) such that \(G\) is a complete graph on two vertices \(s\) and \(t\).
A \emph{two-terminal series-parallel graph} is a two-terminal-graph which can be constructed by a sequence of series and parallel compositions from basic two-terminal series-parallel graphs.
A graph \(G\) is \emph{series-parallel} if one can choose a source \(s\) and a sink \(t\) so that \((G, s, t)\) is a two-terminal series-parallel graph.
The proof of the following characterization of graphs of treewidth at most \(2\) can be found for example in~\cite{brandstaedt-le-spinrad}.
\begin{theorem}\label{thm:tw2-series-parallel}
  A graph has treewidth at most \(2\) if and only if it is a subgraph of a series-parallel graph.
\end{theorem}

In this paper, by \emph{binary tree} we mean a rooted tree \(T\) where every internal node \(u \in V(T)\) has two children: a left child denoted by \(\ell_T(u)\)  and a right child denoted by \(r_T(u)\).
We omit subscript \(T\) when the tree is clear from context.


Let \((G, s, t)\) be a two-terminal series-parallel graph.
An \emph{s-t tree-decomposition} of \((G, s, t)\) is a pair
\((T, \{(Y_u, s_u, t_u) : u \in V(T)\})\) satisfying the following conditions:
\begin{enumerate}
  \item \((T, \{Y_u : u \in V(T)\})\) is a tree-decomposition of \(G\) of width at most \(2\) such that \(T\) is a binary tree,
  \item \(s_u\) and \(t_u\) are distinct elements of \(Y_u\) for \(u \in V(T)\), 
  \item \(s_v = s\) and \(t_v = t\) if \(v\) is the root of \(T\),
  \item for every leaf \(u\) of \(T\), we have \(|Y_u| = 2\),
  \item for every internal node \(u\) of \(T\), if \(|Y_u| = 2\), then \(s_{\ell(u)} = s_{r(u)} = s_u\) and \(t_{\ell(u)} = t_{r(u)} = t_u\), and
  \item for every internal node \(u\) of \(T\), if \(|Y_u| = 3\), then \(s_{\ell(u)} = s_u\), \(t_{\ell(u)} = s_{r(u)} \in Y_u\) and \(t_{r(u)} = t_u\).
\end{enumerate}

\begin{lemma}
  Every two-terminal series-parallel graph \((G, s, t)\) has an s-t tree-decomposition. 
\end{lemma}
\begin{proof}
  We prove the lemma by induction on the number of edges in \(G\).
  If \((G, s, t)\) is a basic two-terminal series-parallel graph, then let \(T\) be a binary tree consisting of one node \(v\), and let \((Y_v, s_v, t_v) = (\{s, t\}, s, t)\).
  Clearly, \((T, \{(Y_u, s_u, t_u) : u \in V(T)\})\) is an s-t tree-decomposition of \((G, s, t)\).

  Suppose that  \((G, s, t)\) is a series or parallel composition of two-terminal series-parallel graphs \((G_1, s_1, t_1)\) and \((G_2, s_2, t_2)\).
  Each of the graphs \(G_1\) and \(G_2\) has less edges than \(G\).
  By induction hypothesis, for \(i \in \{1, 2\}\) there exists an s-t tree-decomposition \((T^i, \{(Y_u^i, s_u^i, t_u^i) : u \in V(T^i)\})\) of \((G_i, s_i, t_i)\). We may assume that the trees \(T^1\) and \(T^2\) are disjoint. Let \(T\) be a tree with a root \(v\) such the subtree rooted at the left child of \(v\) is \(T^1\) and the subtree rooted at the right child of \(v\) is \(T^2\).
  Let \(Y_v = \{s_1, t_1\} \cup \{s_2, t_2\}\), \(s_v = s\) and \(t_v = t\).
  In particular, if \((G, s, t)\) is a series composition of \((G_1, s_1, t_1)\) and \((G_2, s_2, t_2)\), then \(|Y_v| = 3\), and if \((G, s, t)\) is a parallel composition of \((G_1, s_1, t_1)\) and \((G_2, s_2, t_2)\), then \(|Y_v| = 2\).
  For \(i \in \{1, 2\}\) and a node \(u \in V(T^i)\), let \((Y_u, s_u, t_u) = (Y_u^i, s_u^i, t_u^i)\) 
  It is straight-forward to see that \((T, \{(Y_u, s_u, t_u) : u \in V(T)\})\) is an s-t tree-decomposition of \((G, s, t)\). This concludes the inductive proof.
\end{proof}

We consider a binary tree \(T\) as a partially ordered set where for two distinct nodes \(u\) and \(v\) of \(T\), we have \(u < v\) if and only if \(u\) is an ancestor of \(v\).
The \emph{lowest common ancestor} of two nodes \(u\) and \(v\), denoted \(u \wedge v\), is the greatest node \(w\) in \(T\) such that \(w \le u\) and \(w \le v\) in \(T\).
By \(\lein_T\) we denote the linear order on \(V(T)\) defined by the in-order tree traversal of \(T\), that is for any two nodes \(u\) and \(v\) of \(T\) we have \(u \lein_T v\)
if and only if \(r(u \wedge v) \not \le u\) and \(\ell(u \wedge v) \not \le v\) in \(T\).
We omit subscript \(T\) in \(\lein_T\) when the tree is clear from context.

\begin{lemma}\label{lem:wx}
  Let \((T, \{(Y_u, s_u, t_u) : u \in V(T)\})\) be an s-t tree-decomposition of a two-terminal series-parallel graph \((G, s, t)\), let  \(x \in V(G) \setminus \{s, t\}\) and let \(w\) be the least node in \(T\) such that \(x \in Y_w\). Then \(x \not \in \{s_w, t_w\}\).
\end{lemma}
\begin{proof}
  If \(w\) is the root of \(T\), then \(s_w = s\) and \(t_w = t\), so \(x \not \in \{s_w, t_w\}\).
  If \(w\) is not the root of \(T\), then let \(v\) be the parent of \(w\).
  We have \(\{s_w, t_w\} \subseteq Y_v\), so by the choice of \(w\), we have \(x \not \in \{s_w, t_w\}\).
\end{proof}

 \begin{figure}
   \includegraphics{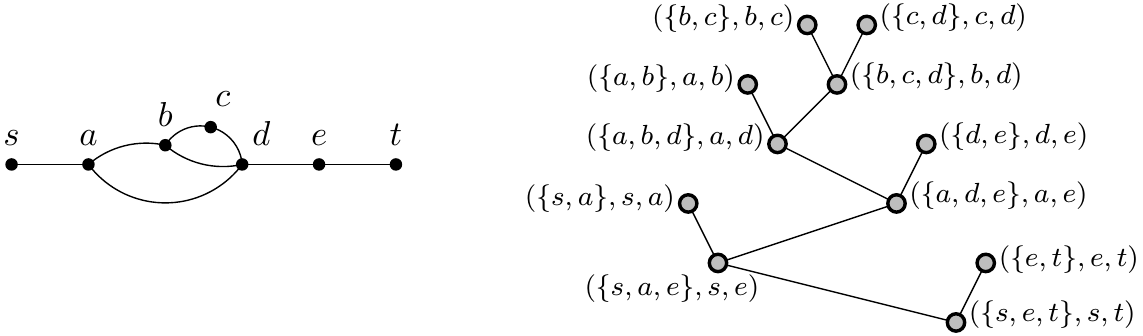}
   \caption{A two-terminal series-parallel graph \((G, s, t)\) and its s-t tree-decomposition.}
 \end{figure}

Given an s-t tree-decomposition \((T, \{(Y_u, s_u, t_u) : u \in V(T)\})\) of a two-terminal series-parallel graph \((G, s, t)\), we can \emph{reverse} it to obtain an s-t tree decomposition of \((G, t, s)\).
The reversed s-t tree-decomposition is \((T', \{(Y'_u, s'_u, t'_u) : u \in V(T')\})\), where \(T'\) is a binary tree with the same nodes having the same children as in \(T\), but for every internal node \(u\), we have \(\ell_{T'}(u) = r_T(u)\) and \(r_{T'}(u) = \ell_T(u)\), and for every node \(u\), we have \((Y'_u, s'_u, t'_u) = (Y_u, t_u, s_u)\).
The partial order defined by the tree \(T'\) is the same as the one defined by \(T\), but the in-order \(\lein_{T'}\) is the reverse of the in-order \(\lein_T\). 

\begin{lemma}\label{lem:decompostion-st-subset}
  Let \((T, \{(Y_u, s_u, t_u) : u \in V(T)\})\) be an s-t tree-decomposition of a two-terminal series-parallel graph \((G, s, t)\), let \(u_1\) and \(u_2\) be two nodes of \(T\), and let \(H\) be a connected subgraph of \(G\) such that \(\{s_{u_1}, t_{u_2}\} \subseteq V(H)\). If \(u_1\) and \(u_2\) are comparable in \(T\), then there exists a node \(v\) on the path between \(u_1\) and \(u_2\) in \(T\) such that \(\{s_v, t_v\} \subseteq V(H)\).
\end{lemma}
\begin{proof}
  Suppose first that \(u_1 \le u_2\) in \(T\).
  Let \(v\) be the greatest node in \(T\) such that \(u_1 \le v \le u_2\) and \(s_v \in V(H)\).
  If \(v = u_2\), then \(\{s_v, t_v\} = \{s_v, t_{u_2}\} \subseteq V(H)\), so the node \(v\) satisfies the lemma.
  Otherwise, let \(v'\) denote the child of \(v\) such that \(u_1 \le v < v' \le u_2\) in \(T\).
  By the choice of \(v\), we have \(s_{v'} \not \in V(H)\).
  In particular \(s_{v'} \neq s_{v}\), which implies \(|Y_{v}| = 3\), \(v' = r(v)\) and \(t_v = t_{v'}\).
  The nodes \(v\) and \(v'\) lie on the path between \(u_1\) and \(u_2\) in \(T\).
  By Observation~\ref{obs:tree-decomposition-separations} with \(u_1 = u_1\), \(u_2 = u_2\), \(v_1 = v\), \(v_2 = v'\) and \(H = H\), we have \(V(H) \cap (Y_{v} \cap Y_{v'}) \neq \emptyset\).
  Since \(Y_{v} \cap Y_{v'} = \{s_{v'}, t_{v'}\}\) and \(s_{v'} \not \in V(H)\), we have \(t_{v'} \in V(H)\).
  Hence \(\{s_v, t_v\} = \{s_{v}, t_{v'}\} \subseteq V(H)\), and the node \(v\) satisfies the lemma.

  Now suppose that \(u_2 < u_1\) in \(T\).
  In the reversed s-t tree-decomposition \((T', \{(Y_u, s'_u, t'_u) : u \in V(T')\})\) we have \(s'_{u_2} = t_{u_2} \in V(H)\) and \(t'_{u_1} = s_{u_1} \in V(H)\).
  By applying the earlier case to \((T', \{(G_u, s'_u, t'_u) : u \in V(T)\})\) and the nodes \(u_2\) and \(u_1\), we obtain a node \(v\) on the path between \(u_1\) and \(u_2\) such that \(\{s'_v, t'_v\} \subseteq V(H)\). Since \(s'_v = t_v\) and \(t'_v = s_v\), this implies \(\{s_v, t_v\} \subseteq V(H)\), as required.
\end{proof}

\section{The proof}

Let \(P\) be a poset with a cover graph of treewidth at most \(2\).
By Theorem~\ref{thm:tw2-series-parallel}, the cover graph of \(P\) is a subgraph of a series-parallel graph.
Let \((G, s, t)\) be a two-terminal series-parallel graph such that \(G\) contains the cover graph of \(P\) as a subgraph.
After replacing \((G, s, t)\) with its series composition with two basic two-terminal series-parallel graphs we assume that neither \(s\) nor \(t\) is an element of \(P\).  
Fix an s-t tree-decomposition \((T, \{(Y_u, s_u, t_u) : u \in V(T)\})\) of \((G, s, t)\).
For every node \(u\) of \(T\), such that \(|Y_u| = 3\), let \(m_u\) denote the only element of \(Y_u \setminus \{s_u, t_u\}\).
For every element \(x\) of \(P\), let \(w(x)\) denote the least node in \(T\) containing \(x\) in its bag. By Lemma~\ref{lem:wx}, for every element \(x\) of \(P\), we have \(|Y_{w(x)}| = 3\) and \(x = m_{w(x)}\). For every pair \((x, y) \in \Inc(P)\), let \(w(x, y) = w(x) \wedge w(y)\).

Let \((x,y) \in \Inc(P)\).
We say \((x, y)\) is of \emph{type \(1\)} if
\[
  U(x) \cap Y_{w(x, y)} = \emptyset \quad \text{or} \quad D(y) \cap Y_{w(x, y)} = \emptyset.
\]
Otherwise, we have
\[
  U(x) \cap Y_{w(x, y)} \neq \emptyset \quad \text{and} \quad D(y) \cap Y_{w(x, y)} \neq \emptyset,
\]
and we say \((x, y)\) is of \emph{type \(2\)}.
Let \(I_1\) and \(I_2\) denote the sets of incomparable pairs in \(P\) of type 1 and type 2, respectively.

By Observation~\ref{obs:reversible-sets-and-dimension}, to prove the theorem
it suffices to show that \(\Inc(P)\) can be partitioned into \(12\) reversible sets.
We will show that \(I_1\) can be partitioned into \(4\) reversible sets and \(I_2\) can be partitioned into \(8\) reversible sets.

For every incomparable pair \((x, y) \in \Inc(P)\), we have \(m_{w(x)} = x\) and \(m_{w(y)} = y\), so \(w(x) \neq w(y)\). Let
\[
\alpha(x, y) =
\begin{cases}
  1&\text{if \(w(x) \ltin w(y)\),}\\
  2&\text{if \(w(y) \ltin w(x)\).}\\
\end{cases}
\]

\begin{claim}\label{clm:alpha-1}
  Let \((x, y) \in \Inc(P)\) be a pair such that \(\alpha(x, y) = 1\), let \(u\) be a node of \(T\), and let \(H\) be a connected subgraph of \(G\) such that \(V(H) \cap Y_u \neq \emptyset\).
  \begin{enumerate}
    \item\label{itm:alpha-1-x} If \(x \in V(H)\) and \(\ell(w(x, y)) \not \le u\) in \(T\), then \(V(H) \cap \{s_{\ell(w(x, y))}, t_{\ell(w(x, y))}\} \neq \emptyset\).
    \item\label{itm:alpha-1-y} If \(y \in V(H)\) and \(r(w(x, y)) \not \le u\) in \(T\), then \(V(H) \cap \{s_{r(w(x, y))}, t_{r(w(x, y))}\} \neq \emptyset\).
  \end{enumerate}
\end{claim}
\begin{proof}
  For the proof of item (\ref{itm:alpha-1-x}), suppose that \(x \in V(H)\) and \(\ell(w(x, y)) \not \le u\) in \(T\).
  Since \(\alpha(x, y) = 1\), either \(w(x) = w(x, y)\) or \(\ell(w(x, y)) \le w(x)\) in \(T\).
  Let \(v\) denote the node \(\ell(w(x))\) if \(w(x) = w(x, y)\), or the node \(w(x)\) if \(\ell(w(x, y)) \le w(x)\) in \(T\).
  Either way, we have \(\ell(w(x, y)) \le v\) in \(T\) and \(x \in Y_v\).
  Since \(\ell(w(x, y)) \le v\) and \(\ell(w(x, y)) \not \le u\) in \(T\), the nodes \(\ell(w(x, y))\) and \(w(x, y)\) lie on the path between \(v\) and \(u\) in \(T\).
  By Observation~\ref{obs:tree-decomposition-separations} with \(u_1 = v\), \(u_2 = u\), \(v_1 = \ell(w(x, y))\), \(v_2 = w(x, y)\) and \(H = H\), we have
  \(V(H) \cap \{s_{\ell(w(x, y))}, t_{\ell(w(x, y))}\} = V(H) \cap (Y_{\ell(w(x, y))} \cap Y_{w(x, y)}) \neq \emptyset\), as required.

  Now suppose that \(y \in V(H)\) and \(r(w(x, y)) \not \le u\) in \(T\).
  In the reversed decomposition \((T', \{(Y'_u, s'_u, t'_u) : u \in V(T')\})\), we have \(\alpha(y, x) = 1\) since \(w(y) \ltin_{T'} w(x)\). Moreover \(\ell_{T'}(w(y, x)) = r_T(w(x, y)) \not \le u\) in \(T'\), so by item (\ref{itm:alpha-1-x}) applied to the pair \((y, x)\), \(u\) and \(H\), we have \(V(H) \cap \{s_{r_T(w(x, y))}, t_{r_T(w(x, y))}\} = V(H) \cap \{s_{\ell_{T'}(w(x, y))}, t_{\ell_{T'}(w(x, y))}\}\neq \emptyset\).
  Item~(\ref{itm:alpha-1-y}) follows.
\end{proof}


\subsection{Pairs of type 1}

For every pair \((x, y) \in I_1\), we define
\[
\beta(x, y) =
\begin{cases}
  1&\text{if \(U(x) \cap Y_{w(x, y)} = \emptyset\),}\\
  2&\text{if \(U(x) \cap Y_{w(x, y)} \neq \emptyset\).}\\
\end{cases}
\]

The \emph{signature} of a pair \((x, y) \in I_1\) is \[\sigma_1(x, y) = (\alpha(x, y), \beta(x, y)).\]

To prove that \(I_1\) can be partitioned into \(4\) reversible sets, it remains to show that every set consisting of incomparable pairs of type \(1\) having the same signature is reversible.

Let \(I \subseteq \Inc(P)\) be a set such that for every \((x, y) \in I\), we have \(\sigma_1(x, y) = (\alpha, \beta)\) for some \((\alpha, \beta) \in \{1,2\}^2\).
We aim to show that \(I\) is reversible.
First we will show that it is enough to consider the case \((\alpha, \beta) = (1,1)\).

\begin{claim}\label{clm:I1-beta-1}
  If \(\beta = 2\), then in the dual poset \(P^d\) with \((T, \{(Y_u, s_u, t_u) : u \in V(T)\})\), each pair from \(I^{-1}\) is a pair of type 1 with signature \((3-\alpha, 3-\beta)\).
\end{claim}
\begin{proof}
  Suppose \(\beta = 2\), and let \((y, x) \in I^{-1}\).
  In \(P\) with \((T, \{(Y_u, s_u, t_u) : u \in V(T)\})\), the pair \((x, y)\) is of type \(1\) and \(\beta(x, y) = 2\), so \(D_P(y) \cap Y_{w(x, y)} = \emptyset\).
  Hence \(U_{P^d}(y) \cap Y_{w(y, x)} \neq \emptyset\). This implies that in \(P^d\) with \((T, \{(Y_u, s_u, t_u) : u \in V(T)\})\), the pair \((y, x)\) is of type \(1\) and \(\beta(y, x) = 1 = 3 - \beta\).
  The value of \(\alpha(y, x)\) does not depend on the poset, just on the s-t tree-decomposition, and \(\alpha(y, x) = 3 - \alpha(x, y) = 3-\alpha\).
\end{proof}

The set \(I\) is reversible in \(P\) if and only if \(I^{-1}\) is reversible in \(P^d\).
Hence by Claim~\ref{clm:I1-beta-1}, after possibly replacing \(P\) with \(P^d\) and \(I\) with \(I^{-1}\), we may assume \(\beta = 1\).

\begin{claim}\label{clm:I1-alpha-1}
  In \(P\) with the reversed s-t tree-decomposition \((T', \{(Y'_u, s'_u, t'_u) : u \in V(T')\})\), each pair from \(I\) is a pair of type \(1\) with signature \((3-\alpha, \beta)\).
\end{claim}
\begin{proof}
  Let \((x, y) \in I\).
  We have \(Y'_{w(x, y)} = Y_{w(x, y)}\), so  in \(P\) with \((T', \{(Y'_u, s'_u, t'_u) : u \in V(T')\})\), the pair \((x, y)\) is of type \(1\), and \(\beta(x, y) = \beta\).
  Moreover, since \(\ltin_{T}\) is the reverse of \(\ltin_{T'}\), we have \(\alpha(x, y) = 3 - \alpha\) in \(P\) with \((T', \{(Y'_u, s'_u, t'_u) : u \in V(T')\})\).
\end{proof}

By Claim~\ref{clm:I1-alpha-1}, after possibly replacing \((G, s, t)\) with \((G, t, s)\) and \((T, \{(Y_u, s_u, t_u) : u \in V(T)\})\) with \((T', \{(Y'_u, s'_u, t'_u) : u \in V(T')\})\), we may assume \((\alpha, \beta) = (1, 1)\).

We proceed to the proof that \(I\) is reversible.
Suppose to the contrary that \(I\) is not reversible.
By Observation~\ref{obs:reversible-sets-characterization}, there is an alternating cycle \(\{(x_i, y_i) : i \in \{1, \dots, n\}\}\) in \(P\) such that  \(\sigma_1(x_i, y_i) = (1, 1)\) for \(i \in \{1, \dots, n\}\).
Without loss of generality, we assume \(w(y_1) \lein w(y_2)\).
Since \(\beta(x_1, y_1) = 1\),  we have
\(
  U(x_1) \cap \{s_{\ell(w(x_1, y_1))}, t_{\ell(w(x_1, y_1))}\} = \emptyset
\).
Hence, by Claim~\ref{clm:alpha-1} item~(\ref{itm:alpha-1-x}) with \((x, y) = (x_1, y_1)\), \(u = w(y_2)\) and \(H = G[U(x_1)]\), 
we have \(\ell(w(x_1, y_1)) \le w(y_2)\) in \(T\).
Therefore \(w(y_2) \ltin w(x_1, y_1) \lein w(y_1)\), which contradicts the assumption \(w(y_1) \lein w(y_2)\).
Hence \(I\) is reversible.
This completes the proof that the set \(I_1\) can be partitioned into \(4\) reversible sets.

\subsection{Pairs of type 2}

For every pair \((x, y) \in I_2\), we define 
\[
\gamma(x, y) =
  \begin{cases}
    1\quad\text{if \(\{s_u, t_u\} \not \subseteq U(x)\) for every \(u \le w(x, y)\) in \(T\),}\\
    2\quad\text{if \(\{s_u, t_u\} \subseteq U(x)\) for some \(u \le w(x, y)\) in \(T\).}\\
  \end{cases}
\]

Let \((x, y) \in I_2\) be a pair such that \(|Y_{w(x, y)}| = 2\).
By definition of type \(2\),
the sets \(U(x) \cap Y_{w(x, y)}\) and \(D(y) \cap Y_{w(x, y)}\) are disjoint non-empty subsets of \(Y_{w(x, y)} = \{s_{w(x, y)}, t_{w(x, y)}\}\).
Hence, either \(s_{w(x, y)} \in U(x)\) and \(t_{w(x, y)} \in D(y)\), or \(t_{w(x, y)} \in U(x)\) and \(s_{w(x, y)} \in D(y)\).
For every pair \((x, y) \in I_2\) we define 
\[
\delta(x, y) = 
  \begin{cases}
    1&\text{if \(|Y_{w(x, y)}| = 2\), \(s_{w(x, y)} \in U(x)\) and \(t_{w(x, y)} \in D(y)\),}\\
    2&\text{if \(|Y_{w(x, y)}| = 2\), \(t_{w(x, y)} \in U(x)\) and \(s_{w(x, y)} \in D(y)\),}\\
    \alpha(x, y)&\text{if \(|Y_{w(x, y)}| = 3\).}
  \end{cases}
\]

The \emph{signature} of a pair \((x, y) \in I_2\) is
\[
  \sigma_2(x, y) = (\alpha(x, y), \gamma(x, y), \delta(x, y)).
\]

To complete the proof of the theorem, it remains to show that every set consisting of pairs of type \(2\) having the same signature is reversible.

Let \((\alpha, \gamma, \delta) \in \{1,2\}^3\) and let \(I \subseteq I_2\) be a set such that for every \((x, y) \in I\), we have \(\sigma_2(x, y) = (\alpha, \gamma, \delta)\).
We aim to show that \(I\) is reversible in \(P\).
First we will show that it is enough to consider the case \((\alpha, \gamma, \delta) = (1,1,1)\).

\begin{claim}\label{clm:no-u-prime}
  For every pair \((x, y) \in \Inc(P)\) there do not exist nodes \(u\) and \(u'\) such that \(u \le w(x, y)\) and \(u' \le w(x, y)\) in \(T\), \(\{s_u, t_u\} \subseteq U(x)\) and \(\{s_{u'}, t_{u'}\} \subseteq D(y)\).
\end{claim}
\begin{proof}
  Toward a contradiction, suppose that there do exist such nodes \(u\) and \(u'\) for a pair \((x, y) \in \Inc(P)\).
  We have \(u \le w(x, y)\) and \(u' \le w(x, y)\) in \(T\), so the nodes \(u\) and \(u'\) are comparable in \(T\).
  Moreover, the nodes \(u\) and \(u'\) are distinct, since \(U(x) \cap D(y) = \emptyset\).
  Suppose first that \(u < u'\) in \(T\), and let \(v\) denote the parent of \(u'\).
  In particular \(u \le v < u' \le w(x, y) \le w(x)\) in \(T\).
  By Observation~\ref{obs:tree-decomposition-separations} with \(u_1 = w(x)\), \(u_2 = u\), \(v_1 = u'\), \(v_2 = v\) and \(H = G[U(x)]\), we have \(U(x) \cap (Y_{u'} \cap Y_v) \neq \emptyset\). But \(Y_{u'} \cap Y_v = \{s_{u'}, t_{u'}\}\), so \(U(x) \cap D(y) \neq \emptyset\), which is a contradiction.
  By a symmetric argument, if \(u' < u\) in \(T\), then the downset \(D(y)\) intersects the set \(\{s_{u}, t_{u}\}\), which is again a contradiction.
\end{proof}

\begin{claim}\label{clm:I2-gamma-1}
  In the dual poset \(P^d\) with \((T, \{(Y_u, s_u, t_u) : u \in V(T)\})\), each pair \((y, x) \in I^{-1}\) is a pair of type \(2\) such that \(\alpha(y, x) = 3 - \alpha\) and \(\delta(y, x) = 3-\delta\). Moreover, if \(\gamma = 2\), then \(\gamma(y, x) = 3-\gamma\) in \(P^d\) with \((T, \{(Y_u, s_u, t_u) : u \in V(T)\})\).
\end{claim}
\begin{proof}
  Let \((y, x) \in I^{-1}\).
  We have \(D_{P^d}(x) \cap Y_{w(y, x)} = U_P(x) \cap Y_{w(x, y)}\) and \(U_{P^d}(y) \cap Y_{w(y, x)} = D_P(y) \cap Y_{w(x, y)}\).
  Since \((x, y)\) is a pair of type \(2\) in \(P\) with \((T, \{(Y_u, s_u, t_u) : u \in V(T)\})\), this implies \((y, x)\) is a pair of type \(2\) in \(P^d\) with \((T, \{(Y_u, s_u, t_u) : u \in V(T)\})\).

  The value of \(\alpha(y, x)\) does not depend on the poset \(P\), just on the s-t tree-decomposition, and \(\alpha(y, x) = 3-\alpha\).

  If \(|Y_{w(y, x)}| = 2\), then we have \(s_{w(x, y)} \in D_P(y) = U_{P^d}(y)\) and \(t_{w(x, y)} \in U_P(x) = D_{P^d}(x)\), so indeed \(\delta(y, x) = 3 - \delta\) in  \(P^d\) with \((T, \{(Y_u, s_u, t_u) : u \in V(T)\})\).
  Suppose \(|Y_{w(x, y)}| = 3\). 
  In such case \(\delta(x, y) = \alpha(x, y)\) in \(P\) with \((T, \{(Y_u, s_u, t_u) : u \in V(T)\})\), so \(\delta = \alpha\).
  Hence in \(P^d\) with \((T, \{(Y_u, s_u, t_u) : u \in V(T)\})\) we have \(\delta(y, x) = \alpha(y, x) = 3-\alpha = 3-\delta\).
  
  Suppose \(\gamma = 2\).
  We have \(\gamma(x, y) = 2\) in \(P\) with \((T, \{(Y_u, s_u, t_u) : u \in V(T)\})\), so by Claim~\ref{clm:no-u-prime}, there does not exist a node \(u'\) of \(T\) such that \(u' \le w(x, y)\) in \(T\) and \(\{s_{u'}, t_{u'}\} \subseteq D_P(y) = U_{P^d}(y)\).
  Hence \(\gamma(y, x) = 1 = 3-\gamma\) in \(P^d\) with \((T, \{(Y_u, s_u, t_u) : u \in V(T)\})\) if \(\gamma = 2\).
\end{proof}

The set \(I^{-1}\) is reversible in \(P^d\) if and only if the set \(I\) is reversible in \(P\).
Hence by Claim~\ref{clm:I2-gamma-1}, after possibly replacing \(P\) with \(P^d\) and \(I\) with \(I^{-1}\), we may assume \(\gamma = 1\).

\begin{claim}\label{clm:I2-delta-1}
  In \(P\) with the reversed s-t tree-decomposition \((T', \{(Y'_u, s'_u, t'_u) : u \in V(T')\})\), each pair from \(I\) is a pair of type \(2\) with signature \((3 - \alpha, \gamma, 3-\delta)\).
\end{claim}
\begin{proof}
  Let \((x, y) \in I\).
  Since \(Y'_{w(x, y)} = Y_{w(x, y)}\), the pair \((x, y)\) is of type \(2\) in \(P\) with \((T', \{(Y'_u, s'_u, t'_u) : u \in V(T')\})\).

  The in-order \(\lein_{T}\) is the reverse of \(\lein_{T'}\), so with \((T', \{(Y'_u, s'_u, t'_u) : u \in V(T')\})\) we have \(\alpha(x, y) = 3-\alpha\).

  The trees \(T\) and \(T'\) differ only with the order of children and for every \(u \in V(T')\) we have \(\{s'_u, t'_u\} = \{s_u, t_u\}\), so the value of \(\gamma(x, y)\) remains the same with the reversed decomposition.

  If \(|Y_{w(x, y)}| = 2\), then \(s'_{w(x, y)} = t_{w(x, y)} \in U(x)\) and \(t'_{w(x, y)} = s_{w(x, y)} \in D(y)\), so indeed \(\delta(x, y) = 3 - \delta\).
  Suppose \(|Y_{w(x, y)}| = 3\).
  In \(P\) with \((T, \{(Y_u, s_u, t_u) : u \in V(T)\})\) we have \(\delta(x, y) = \alpha(x, y)\), so \(\delta = \alpha\).
  Hence in \(P\) with \((T', \{(Y'_u, s'_u, t'_u) : u \in V(T')\})\), we have \(\delta(x, y) = \alpha(x, y) = 3-\alpha = 3-\delta\).
\end{proof}

By Claim~\ref{clm:I2-delta-1}, after possibly replacing \((G, s, t)\) with \((G, t, s)\) and \((T, \{(Y_u, s_u, t_u) : u \in V(T)\})\) with \((T', \{(Y_u', s_u', t_u') : u \in V(T')\})\), we may assume that \(\gamma = 1\) and \(\delta = 1\).

Let \(T''\) denote a tree obtained from \(T\) by swapping the left and right children of every node \(u\) such that \(|Y_u| = 2\).
It is easy to observe that the pair \((T'', \{(Y_u, s_u, t_u) : u \in V(T'')\})\) is again an s-t tree-decomposition of \((G, s, t)\).

\begin{claim}\label{clm:I2-alpha-1}
  If \(\alpha = 2\) and \(\delta = 1\), then in \(P\) with the s-t tree-decomposition \((T'', \{(Y_u, s_u, t_u) : u \in V(T'')\})\), each pair from \(I\) is a pair of type \(2\) with signature \((3-\alpha, \gamma, \delta)\).
\end{claim}
\begin{proof}
  Suppose that \(\alpha = 2\) and \(\delta = 1\), and let \((x, y) \in I\).
  Since the set \(Y_{w(x, y)}\) remains the same in the new decomposition, the pair \((x, y)\) is of type \(2\) in \(P\) with \((T'', \{(Y_u, s_u, t_u) : u \in V(T'')\})\).

  Since \(\alpha(x, y) \neq \delta(x, y)\), \(|Y_{w(x, y)}| = 2\).
  Hence \(w(x) \neq w(x, y)\) and \(w(y) \neq w(x, y)\).
  Therefore \(\ell_{T''}(w(x, y)) = r_T(w(x, y)) \le w(x)\) and \(r_{T''}(w(x, y)) = \ell_T(w(x, y)) \le w(y)\) in \(T'' \).
  Hence, in \(P\) with \((T'', \{(Y_u, s_u, t_u) : u \in V(T'')\})\), we have \(\alpha(x, y) = 1 = 3-\alpha\).

  Since the values \(\gamma(x, y)\) and \(\delta(x, y)\) do not depend on the order of the children of the nodes with bags of size \(2\), we have \(\gamma(x, y) = \gamma\) and \(\delta(x, y) = \delta\) in \(P\) with \((T'', \{(Y_u, s_u, t_u) : u \in V(T'')\})\).
\end{proof}

By Claim~\ref{clm:I2-alpha-1}, after possibly replacing the tree \(T\) with \(T''\), we may assume \((\alpha,\gamma,\delta) = (1,1,1)\).

We proceed to the proof that \(I\) is reversible.
Suppose to the contrary that \(I\) is not reversible.
By Observation~\ref{obs:reversible-sets-characterization}, there exists a strict alternating cycle \(\{(x_i, y_i) : i \in \{1, \dots, n\}\}\) in \(P\) such that for every \(i \in \{1, \dots, n\}\), the pair \((x_i, y_i)\) is of type \(2\) and \(\sigma_2(x_i, y_i)  = (1, 1, 1)\).

\begin{claim}\label{clm:w-nae}
  Not all nodes \(w(x_i, y_i)\) for \(i \in \{1, \dots, n\}\) are the same.
\end{claim}

\begin{proof}
  Suppose to the contrary that all nodes \(w(x_i, y_i)\) for \(i \in \{1, \dots, n\}\) are the same node~\(w\).
  We break the argument into two cases.
  
  \textit{Case 1:}
  Suppose that \(|Y_w| = 2\).
  Let \(Q\) be a covering chain from \(x_1\) to \(y_2\).
  Since \(\alpha(x_2, y_2) = 1\), we have \(\ell(w) \not \le w(y_2)\) in \(T\), so by Claim~\ref{clm:alpha-1} item~(\ref{itm:alpha-1-x}) with \((x, y) = (x_1, y_1)\), \(u = w(y_2)\) and \(H = G[Q]\), we have \(Q \cap \{s_{\ell(w)}, t_{\ell(w)}\} \neq \emptyset\).
  We have \(|Y_w| = 2\), so \(s_w \in Q\) or \(t_w \in Q\).
  If \(s_w \in Q\), then \(Q \cap U(x_2) \neq \emptyset\), since \(\delta(x_2, y_2) = 1\).
  This contradicts \((x_2, y_2) \in \Inc(P)\).
  Similarly, if \(t_w \in Q\), then \(Q \cap D(y_1) \neq \emptyset\), since \(\delta(x_1, y_1) = 1\).
  This contradicts \((x_1, y_1) \in \Inc(P)\).

  \textit{Case 2:}
  Suppose that \(|Y_w| = 3\).
  For \(i \in \{1, 2\}\), let \(Q_i\) be a covering chain from \(x_i\) to \(y_{i+1}\).
  Since \((x_2, y_2) \in \Inc(P)\), the chains \(Q_1\) and \(Q_2\) are disjoint.
  Hence at most one of the chains contains \(m_w\).
  Let \(i \in \{1, 2\}\) be such that \(m_w \not \in Q_i\).
  Since \(\alpha(x_{i+1}, y_{i+1}) = 1\), we have \(\ell(w) \not \le w(y_{i+1})\) in \(T\), so by Claim~\ref{clm:alpha-1} item~(\ref{itm:alpha-1-x}) with \((x, y) = (x_i, y_i)\), \(u = w(y_{i+1})\) and \(H = G[Q_i]\), we have \(Q_i \cap \{s_{\ell(w)}, t_{\ell(w)}\} \neq \emptyset\).
  But \(t_{\ell(w)} = m_w \not \in Q_i\).
  This implies \(s_{\ell(w)} \in Q_i\), so \(s_w \in U(x_i)\).
  Similarly, since \(\alpha(x_i, y_i) = 1\), we have \(r(w) \not \le w(x_i)\) in \(T\), so by Claim~\ref{clm:alpha-1} item~(\ref{itm:alpha-1-y}) with \((x, y) = (x_{i+1}, y_{i+1})\), \(u = w(x_i)\) and \(H = G[Q_i]\), we have \(Q_i \cap \{s_{r(w)}, t_{r(w)}\} \neq \emptyset\).
  We have \(s_{r(w)} = m_w \not \in Q_i\), so \(t_{r(w)} \in Q_i\), and thus \(t_w \in U(x_i)\).
  Hence \(\{s_w, t_w\} \subseteq U(x_i)\), which contradicts \(\gamma(x_i, y_i) = 1\).
\end{proof}

By Claim~\ref{clm:w-nae}, we may (and do) assume without loss of generality
\[
  w(x_1, y_1) \ltin w(x_2, y_2).
\]

\begin{claim}\label{clm:existsw}
  Let \((x, y)\) and \((x', y')\) be two pairs of type \(2\) such that \(\alpha(x, y) = \alpha(x', y') = 1\), \(\delta(x, y) = \delta(x', y') = 1\) and \(w(x, y) \ltin w(x', y')\), and let \(Q\) be a covering chain from \(x\) to \(y'\). There exists a node \(w \in \{w(x, y), \ell(w(x, y)), r(w(x, y))\}\) such that \(s_w \in Q\) and \(t_w \not \in Q\).
\end{claim}
\begin{proof}
  We will first show that \(\ell(w(x, y)) \not \le w(y')\) in \(T\).
  Suppose to the contrary that \(\ell(w(x, y)) \le w(y')\) in \(T\).
  We have \(w(x, y) \le \ell(w(x, y)) \le w(y')\) and \(w(x', y') \le w(y')\) in \(T\), so the nodes \(w(x, y)\) and \(w(x', y')\) are comparable in \(T\).
  Since \(w(x, y) \ltin w(x', y')\), either \(r(w(x, y)) \le w(x', y')\) in \(T\) or \(\ell(w(x', y')) \le w(x, y)\) in \(T\).
  In the former case we have \(r(w(x, y)) \le w(x', y') \le w(y')\) in \(T\), which contradicts \(\ell(w(x, y)) \le w(y')\) in \(T\).
  In the latter case we have \(\ell(w(x', y')) \le w(x, y) \le w(y')\) in \(T\), which contradicts \(\alpha(x', y') = 1\).
  Hence indeed \(\ell(w(x, y)) \not \le w(y')\) in \(T\).

  By Claim~\ref{clm:alpha-1} item (\ref{itm:alpha-1-x}) with \((x, y) = (x, y)\), \(u = w(y')\) and \(H = G[Q]\), we have \(Q \cap \{s_{\ell(w(x, y))}, t_{\ell(w(x, y))}\} \neq \emptyset\).
  Let \(z_1 \in Q \cap \{s_{\ell(w(x, y))}, t_{\ell(w(x, y))}\}\).
  
  By Claim~\ref{clm:alpha-1} item (\ref{itm:alpha-1-y}) with \((x, y) = (x, y)\), \(u = w(x, y)\) and \(H = G[D(y)]\), we have \(D(y) \cap \{s_{r(w(x, y))}, t_{r(w(x, y))}\} \neq \emptyset\).
  Let \(z_2 \in D(y) \cap \{s_{r(w(x, y))}, t_{r(w(x, y))}\}\).
  We have \(z_2 \not \in Q\) since otherwise \(x \le z_2 \le y\), contradicting \((x, y) \in \Inc(P)\).

  Suppose that \(|Y_{w(x, y)}| = 2\).
  The vertices \(z_1\) and \(z_2\) are distinct elements of \(\{s_{w(x, y)}, t_{w(x, y)}\}\), so \(z_1 = s_{w(x, y)}\) and \(z_2 = t_{w(x, y)}\) since \(\delta(x, y) = 1\).
  Hence the node \(w = w(x, y)\) satisfies the claim.

  Now suppose that \(|Y_{w(x, y)}| = 3\).
  We have \(z_1 \in \{s_{w(x, y)}, m_{w(x, y)}\}\) and \(z_2 \in \{m_{w(x, y)}, t_{w(x, y)}\}\).

  If \(z_1 = s_{w(x, y)}\) and \(z_2 = m_{w(x, y)}\), then \(w = \ell(w(x, y))\) satisfies the claim.

  If \(z_1 = s_{w(x, y)}\) and \(z_2 = t_{w(x, y)}\), then \(w = w(x, y)\) satisfies the claim.

  Finally, if \(z_1 = m_{w(x, y)}\) and \(z_2 = t_{w(x, y)}\), then \(w = r(w(x, y))\) satisfies the claim.
\end{proof}

Let \(Q\) be a covering chain from \(x_1\) to \(y_2\).
By Claim~\ref{clm:existsw} with \((x, y) = (x_1, y_1)\), \((x', y') = (x_2, y_2)\) and \(Q = Q\), there exists a node \(w_1\in \{w(x_1, y_1), \ell(w(x_1, y_1)), r(w(x_1, y_1))\}\) such that
\[
  s_{w_1} \in Q \quad \text{and} \quad t_{w_1} \not \in Q.
\]

By Claim~\ref{clm:I2-gamma-1} and Claim~\ref{clm:I2-delta-1}, in the dual poset \(P^d\) with the reversed s-t tree-decomposition \((T', \{(Y'_u, s'_u, t'_u) : u \in V(T')\})\), we have \(\alpha(y_2, x_2) = \alpha(y_1, x_1) = (3 - (3 - \alpha)) = 1\) and \(\delta(y_2, x_2) = \delta(y_1, x_1) = (3 - (3 - \delta)) = 1\).
Moreover \(w(y_2, x_2) \ltin_{T'} w(y_1, x_1)\).
By applying Claim~\ref{clm:existsw} to the dual poset and the reversed decomposition with \((x, y) = (y_2, x_2)\), \((x', y') = (y_1, x_1)\) and \(Q = Q\), we obtain a node \(w_2\in \{w(x_2, y_2), \ell(w(x_2, y_2)), r(w(x_2, y_2))\}\) such that 
\[
  t_{w_2} \in Q \quad \text{and} \quad s_{w_2} \not \in Q
\]
since \(t_{w_2} = s'_{w_2}\) and \(s_{w_2} = t'_{w_2}\).

\begin{claim}\label{clm:tuniq}
  \(t_u \not \in Q\) for \(u \le w_1\) in \(T\).
\end{claim}
\begin{proof}
  Suppose to the contrary that there exists a node \(u\) such that \(u \le w_1\) in \(T\) and \(t_u \in Q\).
  By Lemma~\ref{lem:decompostion-st-subset} with \(u_1=w_1\), \(u_2=u\) and \(H=G[Q]\), there exists a node \(v\) such that \(u \le v \le w_1\) and \(\{s_v, t_v\} \subseteq Q \subseteq U(x_1)\).
  Since \(t_{w_1} \not \in Q\), we have \(v < w_1\) in \(T\), and thus \(v \le w(x_1, y_1)\) in \(T\), contradicting \(\gamma(x_1, y_1) = 1\).
\end{proof}

\begin{claim}
  There exists a node \(v\) such that \(v \not \le w(x_1, y_1)\) in \(T\), \(v \le w(x_2, y_2)\) in \(T\), and \(\{s_v, t_v\} \subseteq Q\).
\end{claim}
\begin{proof}
  Suppose first that \(w_1 \le w_2\) in \(T\). 
  By Lemma~\ref{lem:decompostion-st-subset} with \(u_1 = w_1\), \(u_2 = w_2\) and \(H = G[Q]\), there is a node \(v\) such that \(w_1 \le v \le w_2\) in \(T\) and \(\{s_v, t_v \} \subseteq Q\).
  Since \(t_{w_1} \not \in Q\), we have \(w_1 < v\) in \(T\).
  The node \(w(x_1, y_1)\) is \(w_1\) or its parent, so \(w(x_1, y_1) < v\) in \(T\).
  Since \(s_{w_1} \not \in Q\), we have \(v < w_2\) in \(T\).
  The node \(w(x_2, y_2)\) is \(w_2\) or its parent, so \(v \le w(x_2, y_2)\) in \(T\), and thus the node \(v\) satisfies the claim.

  If \(w_2 \le w_1\) in \(T\), then by Claim~\ref{clm:tuniq}, we have \(t_{w_2} \not \in Q\), which is impossible.

  Finally suppose that \(w_1\) and \(w_2\) are incomparable in \(T\).
  We claim that \(w_1 \ltin w_2\).
  Suppose to the contrary that \(w_2 \ltin w_1\).
  Since \(w_1\) and \(w_2\) are incomparable in \(T\), we have \(\ell(w_1 \wedge w_2) \le w_2\) and \(r(w_1 \wedge w_2) \le w_1\) in \(T\).
  Since \(w(x_i, y_i)\) is \(w_i\) or its parent for \(i \in \{1, 2\}\), we have \(w(x_2, y_2) \lein w_1 \wedge w_2 \lein w(x_1, y_1)\) in \(T\), which is a contradiction.
  Hence indeed \(w_1 \ltin w_2\), and thus \(\ell(w_1 \wedge w_2) \le w_1\) and \(\ell(w_1 \wedge w_2) \not \le w_2\) in \(T\).
  By Observation~\ref{obs:tree-decomposition-separations} with \(u_1= w_1\), \(u_2 = w_2\), \(v_1 = \ell(w_1 \wedge w_2)\), \(v_2 = w_1 \wedge w_2\) and \(H = G[Q]\), we have
  \(Q \cap \{s_{\ell(w_1 \wedge w_2)}, t_{\ell(w_1 \wedge w_2)}\} \neq \emptyset\).
  Hence by Claim~\ref{clm:tuniq}, we have \(s_{w_1 \wedge w_2} = s_{\ell(w_1 \wedge w_2)} \in Q\).
  By Lemma~\ref{lem:decompostion-st-subset} with \(u_1 = w_1 \wedge w_2\), \(u_2 = w_2\) and \(H = G[Q]\), there exists a node \(v\) such that \(w_1 \le v \le w_2\) and \(\{s_v, t_v\} \subseteq Q\).
  Since \(s_{w_2} \not \in Q\), we have \(v < w_2\) in \(T\) and thus \(v \le w(x_2, y_2)\) in \(T\).
  We have \(t_v \in Q\), so by Claim~\ref{clm:tuniq}, \(v \not \le w_1\) in \(T\), so \(v \not \le w(x_1, y_1)\) in \(T\).
  Hence the node \(v\) satisfies the claim.
\end{proof}

\begin{claim}\label{clm:trap}
  For every \(i \in \{1, \dots, n\}\), we have \(v \le w(x_i, y_i)\) in \(T\).
\end{claim}

\begin{proof}
  To prove the claim it is enough to show that for every \(i \in \{2, \dots, n\}\), if \(v \le w(x_i, y_i)\) in \(T\), then \(v \le w(x_{i+1}, y_{i+1})\) in \(T\).
  Let \(i \in \{2, \dots, n\}\) and suppose that \(v \le w(x_i, y_i)\) in \(T\).
  Since \(v \not \le w(x_1, y_1)\) in \(T\), \(v\) is not the root of \(T\).
  Let \(v'\) denote the parent of \(v\).
  We have \(x_i \in D(y_{i+1}) \cap Y_{w(x_i)}\) and since \((x_{i+1}, y_{i+1})\) is of type \(2\), \(D(y_{i+1}) \cap Y_{w(x_{i+1}, y_{i+1})} \neq \emptyset\).
  Moreover, by strictness of the alternating cycle we have \(D(y_{i+1}) \cap U(x_1) = \emptyset\), so \(D(y_{i+1}) \cap \{s_v, t_v\} = \emptyset\).
  Hence, by Observation~\ref{obs:tree-decomposition-separations} with \(u_1 = w(x_i)\), \(u_2 = w(x_{i+1}, y_{i+1})\), \(v_1 = v\), \(v_2 = v'\) and \(H=G[D(y_{i+1})]\), the edge \(v v'\) does not lie on the path between \(w(x_i)\) and \(w(x_{i+1}, y_{i+1})\) in \(T\).
  Since \(v \le w(x_{i}, y_{i}) \le w(x_i)\) in \(T\), this implies \(v \le w(x_{i+1}, y_{i+1})\) in \(T\), as required.
  The claim follows.
\end{proof}

By Claim~\ref{clm:trap}, we have \(v \le w(x_1, y_1)\) in \(T\).
The obtained contradiction completes the proof that the set \(I\) is reversible.
Hence the proof of Theorem~\ref{thm:main} is complete. 

\section*{Acknowledgments}

I thank Piotr Micek for many helpful discussions.

\bibliographystyle{plain}
\bibliography{new-treewidth-2-dimension}

\begin{thebibliography}{1}

\bibitem{biro-keller-young}
Csaba Bir\'{o}, Mitchel~T. Keller, and Stephen~J. Young.
\newblock Posets with cover graph of pathwidth two have bounded dimension.
\newblock {\em Order}, 33(2):195--212, 2016.

\bibitem{brandstaedt-le-spinrad}
Andreas Brandst\"{a}dt, Van~Bang Le, and Jeremy~P. Spinrad.
\newblock {\em Graph classes: a survey}.
\newblock SIAM Monographs on Discrete Mathematics and Applications. Society for
  Industrial and Applied Mathematics (SIAM), Philadelphia, PA, 1999.

\bibitem{felsner-trotter-wiechert}
Stefan Felsner, William~T. Trotter, and Veit Wiechert.
\newblock The dimension of posets with planar cover graphs.
\newblock {\em Graphs Combin.}, 31(4):927--939, 2015.

\bibitem{joret-et-al}
Gwena\"{e}l Joret, Piotr Micek, William~T. Trotter, Ruidong Wang, and Veit
  Wiechert.
\newblock On the dimension of posets with cover graphs of treewidth 2.
\newblock {\em Order}, 34(2):185--234, 2017.

\bibitem{kelly}
David Kelly.
\newblock On the dimension of partially ordered sets.
\newblock {\em Discrete Math.}, 35:135--156, 1981.

\bibitem{trotter-moore}
William~T. Trotter, Jr. and John~I. Moore, Jr.
\newblock The dimension of planar posets.
\newblock {\em J. Combinatorial Theory Ser. B}, 22(1):54--67, 1977.

\bibitem{wiechert}
Veit Wiechert.
\newblock {\em Cover Graphs and Order Dimension}.
\newblock PhD thesis, Technical University of Berlin, 2017.
\newblock
  \url{https://depositonce.tu-berlin.de/bitstream/11303/6248/5/wiechert_veit.pdf}.

\end{thebibliography}
\end{document}